\documentclass[12pt]{amsart}
\usepackage{amsmath, amsthm, amscd, amsfonts, amssymb, graphicx, color,float,pgf,tikz}
\usepackage{amssymb,fontenc}
\usepackage{latexsym,wasysym,mathrsfs}
\usepackage{hyperref}
\usetikzlibrary{arrows}
\usepackage[square,numbers,sort&compress]{natbib}

\makeatletter \oddsidemargin.9375in \evensidemargin \oddsidemargin
\newtheorem{theorem}{Theorem}[section]
\newtheorem{lemma}[theorem]{Lemma}
\newtheorem{proposition}[theorem]{Proposition}
\newtheorem{corollary}[theorem]{Corollary}

\theoremstyle{definition}
\newtheorem{definition}[theorem]{Definition}
\newtheorem{example}[theorem]{Example}

\numberwithin{equation}{section}

\newcommand{\be}{\begin{equation}}
\newcommand{\ee}{\end{equation}}

\numberwithin{equation}{section}

\usepackage{etoolbox}
\makeatletter
\patchcmd{\@settitle}{\uppercasenonmath\@title}{}{}{}
\patchcmd{\@setauthors}{\MakeUppercase}{}{}{}
\makeatother

\allowdisplaybreaks
\begin{document}

\setcounter{page}{1}
\title[Controlled Integral Frames for Hilbert $C^{\ast}$-Modules]{Controlled Integral Frames for Hilbert $C^{\ast}$-Modules}

\author[Hatim LABRIGUI$^{*}$ and Samir KABBAJ]{Hatim LABRIGUI$^1$$^{\ast}$ \MakeLowercase{and} Samir KABBAJ$^1$}

\address{$^{1}$Department of Mathematics, Ibn Tofail University, B.P. 133, Kenitra, Morocco}
\email{\textcolor[rgb]{0.00,0.00,0.84}{  hlabrigui75@gmail.com;  samkabbaj@yahoo.fr}}

\subjclass[2010]{42C15,41A58}

\keywords{Integral Frames, Integral $\ast$-frame, Controlled Integral Frames, controlled Integral $\ast$-frame, $C^{\ast}$-algebra, Hilbert $\mathcal{A}$-modules.\\
\indent
\\
\indent $^{*}$ Corresponding author}
\maketitle
\begin{abstract}
The notion of controlled frames for Hilbert spaces were introduced by Balazs, Antoine and Grybos to improve the numerical efficiency of iterative algorithms for inverting the frame operator. Controlled Frame Theory has a great revolution in recent years. This Theory have been extended from Hilbert spaces to Hilbert  $C^{\ast}$-modules. In this paper we introduce and study the extension of this notion to integral frame for Hilbert $C^{\ast}$-module. Also we give some characterizations between integral frame in Hilbert $C^{\ast}$-module.
\end{abstract}
\section{Introduction and preliminaries}
The concept of frames in Hilbert spaces has been introduced by Duffin and Schaeffer \cite{Duf} in 1952 to study some deep problems in nonharmonic Fourier series. After the fundamental paper \cite{13} by Daubechies, Grossman and Meyer, frames theory began to be widely used, particularly in the more specialized context of wavelet frames and Gabor frames \cite{Gab}.

Hilbert $C^{\ast}$-module arose as generalization of the  Hilbert space notion. The basic idea was to consider modules over $C^{\ast}$-algebras instead of linear spaces and to allow the inner product to take values in the $C^{\ast}$-algebras \cite{28}.
Continuous frames defined by Ali, Antoine and Gazeau \cite{STAJP}. Gabardo and Han in \cite{14} called these kinds frames or frames associated with measurable spaces. For more details, the reader can refer to \cite{MR1}, \cite{MR2} and \cite{ARAN}.\\
The goal of this article is the introduction and the study of the concept of Controlled integral frames for Hilbert $C^{\ast}$-module. Also we give some characterizations between integral frame in Hilbert $C^{\ast}$-module, and we give some characterizations.

In the following we briefly recall the definitions and basic properties of $C^{\ast}$-algebra and Hilbert $\mathcal{A}$-modules. Our references for $C^{\ast}$-algebras are \cite{{Dav},{Con}}. For a $C^{\ast}$-algebra $\mathcal{A}$, if $a\in\mathcal{A}$ is positive we write $a\geq 0$ and $\mathcal{A}^{+}$ denotes the set of positive elements of $\mathcal{A}$.
\begin{definition}\cite{Con}.	
	Let $ \mathcal{A} $ be a unital $C^{\ast}$-algebra and $\mathcal{H}$ be a left $ \mathcal{A} $-module, such that the linear structures of $\mathcal{A}$ and $ \mathcal{H} $ are compatible. $\mathcal{H}$ is a pre-Hilbert $\mathcal{A}$-module if $\mathcal{H}$ is equipped with an $\mathcal{A}$-valued inner product $\langle.,.\rangle_{\mathcal{A}} :\mathcal{H}\times\mathcal{H}\rightarrow\mathcal{A}$, such that is sesquilinear, positive definite and respects the module action. In the other words,
	\begin{itemize}
		\item [(i)] $ \langle x,x\rangle_{\mathcal{A}}\geq0 $, for all $ x\in\mathcal{H} $, and $ \langle x,x\rangle_{\mathcal{A}}=0$ if and only if $x=0$.
		\item [(ii)] $\langle ax+y,z\rangle_{\mathcal{A}}=a\langle x,z\rangle_{\mathcal{A}}+\langle y,z\rangle_{\mathcal{A}},$ for all $a\in\mathcal{A}$ and $x,y,z\in\mathcal{H}$.
		\item[(iii)] $ \langle x,y\rangle_{\mathcal{A}}=\langle y,x\rangle_{\mathcal{A}}^{\ast} $, for all $x,y\in\mathcal{H}$.
	\end{itemize}	 
	For $x\in\mathcal{H}, $ we define $||x||=||\langle x,x\rangle_{\mathcal{A}}||^{\frac{1}{2}}$. If $\mathcal{H}$ is complete with $||.||$, it is called a Hilbert $\mathcal{A}$-module or a Hilbert $C^{\ast}$-module over $\mathcal{A}$.\\
	 For every $a$ in $C^{\ast}$-algebra $\mathcal{A}$, we have $|a|=(a^{\ast}a)^{\frac{1}{2}}$ and the $\mathcal{A}$-valued norm on $\mathcal{H}$ is defined by $|x|=\langle x, x\rangle_{\mathcal{A}}^{\frac{1}{2}}$, for all $x\in\mathcal{H}$.
	
	Let $\mathcal{H}$ and $\mathcal{K}$ be two Hilbert $\mathcal{A}$-modules, a map $T:\mathcal{H}\rightarrow\mathcal{K}$ is said to be adjointable if there exists a map $T^{\ast}:\mathcal{K}\rightarrow\mathcal{H}$ such that $\langle Tx,y\rangle_{\mathcal{A}}=\langle x,T^{\ast}y\rangle_{\mathcal{A}}$ for all $x\in\mathcal{H}$ and $y\in\mathcal{K}$.
	
We reserve the notation $End_{\mathcal{A}}^{\ast}(\mathcal{H},\mathcal{K})$ for the set of all adjointable operators from $\mathcal{H}$ to $\mathcal{K}$ and $End_{\mathcal{A}}^{\ast}(\mathcal{H},\mathcal{H})$ is abbreviated to $End_{\mathcal{A}}^{\ast}(\mathcal{H})$.

\end{definition}

The following lemmas will be used to prove our mains results.
\begin{lemma} \label{l1} \cite{Pas}.
	Let $\mathcal{H}$ be a Hilbert $\mathcal{A}$-module. If $T\in End_{\mathcal{A}}^{\ast}(\mathcal{H})$, then $$\langle Tx,Tx\rangle_{\mathcal{A}}\leq\|T\|^{2}\langle x,x\rangle_{\mathcal{A}}, \qquad x\in\mathcal{H}.$$
\end{lemma}

\begin{lemma} \label{l2} \cite{Ara}.
	Let $\mathcal{H}$ and $\mathcal{K}$ be two Hilbert $\mathcal{A}$-modules and $T\in End_{\mathcal{A}}^{\ast}(\mathcal{H},\mathcal{K})$. Then the following statements are equivalent:
	\begin{itemize}
		\item [(i)] $T$ is surjective.
		\item [(ii)] $T^{\ast}$ is bounded below associted to norm, i.e., there is $m>0$ such that $ m\|x\|\leq \|T^{\ast}x\|$, for all $x\in\mathcal{K}$.
		\item [(iii)] $T^{\ast}$ is bounded below associted to the inner product, i.e., there is $m'>0$ such that $ m'\langle x,x\rangle_{\mathcal{A}} \leq \langle T^{\ast}x,T^{\ast}x\rangle_{\mathcal{A}}$, for all $x\in\mathcal{K}$.
	\end{itemize}
\end{lemma}
\begin{lemma}\cite{Deh}\label{l3}
	Let $\mathcal{H}$ and $\mathcal{K}$ be two Hilbert $\mathcal{A}$-modules and $T\in End_{\mathcal{A}}^{\ast}(\mathcal{H},\mathcal{K})$.
	\begin{itemize}
		\item [(i)] If $T$ is injective and $T$ has closed range, then the adjointable map $T^{\ast}T$ is invertible and $$\|(T^{\ast}T)^{-1}\|^{-1}I_{\mathcal{H}}\leq T^{\ast}T\leq\|T\|^{2}I_{\mathcal{H}}.$$
		\item  [(ii)]	If $T$ is surjective, then the adjointable map $TT^{\ast}$ is invertible and $$\|(TT^{\ast})^{-1}\|^{-1}I_{\mathcal{K}}\leq TT^{\ast}\leq\|T\|^{2}I_{\mathcal{K}}.$$
	\end{itemize}	
\end{lemma}
\begin{lemma} \label{l4} \cite{33}.
	Let $(\Omega,\mu )$ be a measure space, $X$ and $Y$ are two Banach spaces, $\lambda : X\longrightarrow Y$ be a bounded linear operator and $f : \Omega\longrightarrow X$ measurable function; then, 
	\begin{equation*}
	\lambda (\int_{\Omega}fd\mu)=\int_{\Omega}(\lambda f)d\mu.
	\end{equation*}
\end{lemma}
\begin{theorem}\cite{Ch}\label{t0}
Let $X$ be a Banach space, $U : X \longrightarrow X$ a bounded operator and $\|I-U\|<1$. Then $U$ is invertible.
\end{theorem}

\section{Controlled Integral Frames in Hilbert $C^{\ast}$-Modules}
Let $X$ be a Banach space, $(\Omega,\mu)$ a measure space, and $f:\Omega\to X$ be a measurable function. Integral of Banach-valued function $f$ has been defined by Bochner and others. Most properties of this integral are similar to those of the integral of real-valued functions (see \cite{32, 33}). Since every $C^{\ast}$-algebra and Hilbert $C^{\ast}$-module are Banach space, we can use this integral and its properties.

Let $(\Omega,\mu)$ be a measure spaces, $\mathcal{H}$ and $\mathcal{K}$ be two Hilbert $C^{\ast}$-modules over a unital $C^{\ast}$-algebra and $\{\mathcal{H}_{w}\}_{w\in\Omega}$  is a family of submodules of $\mathcal{H}$.  $End_{\mathcal{A}}^{\ast}(\mathcal{H},\mathcal{H}_{w})$ is the collection of all adjointable $\mathcal{A}$-linear maps from $\mathcal{H}$ into $\mathcal{H}_{w}$.

We define the following:
\begin{equation*}
l^{2}(\Omega, \{\mathcal{H}_{w}\}_{\omega \in \Omega})=\left\{x=\{x_{w}\}_{w\in\Omega}: x_{w}\in \mathcal{H}_{w}, \left\|\int_{\Omega}\langle x_{w},x_{w}\rangle_{\mathcal{A}} d\mu(w)\right\|<\infty\right\}.
\end{equation*}
For any $x=\{x_{w}\}_{w\in\Omega}$ and $y=\{y_{w}\}_{w\in\Omega}$, the $\mathcal{A}$-valued inner product is defined by $\langle x,y\rangle_{\mathcal{A}}=\int_{\Omega}\langle x_{w},y_{w}\rangle_{\mathcal{A}} d\mu(w)$ and the norm is defined by $\|x\|=\|\langle x,x\rangle_{\mathcal{A}}\|^{\frac{1}{2}}$. \\
In this case, the $l^{2}(\Omega,\{\mathcal{H}_{w}\}_{\omega \in \Omega})$ is a Hilbert $C^{\ast}$-module (see \cite{28}).
	Let $GL^{+}(\mathcal{H})$ be the set of all positive bounded linear invertible operators on $\mathcal{H}$ with bounded inverse.
\begin{definition}\cite{Ros1}
Let $\mathcal{H}$ be a Hilbert $\mathcal{A}$-module and $(\Omega , \mu)$ a measure space. A mapping $F: \Omega \longrightarrow \mathcal{H}$ is called an integral frame associted to $(\Omega , \mu)$ if: \\
\begin{itemize}
\item [$\centerdot$ ] For all $x\in \mathcal{H}$, $w \longrightarrow \langle x,F_{\omega}\rangle_{\mathcal{A}} $ is measurable function on $\Omega$.
\item [$\centerdot$ ] There is a pair of constants $0<A, B$ such that,
\begin{equation}
A\langle x,x\rangle_{\mathcal{A}} \leq\int_{\Omega}\langle x,F_{\omega}\rangle_{\mathcal{A}} \langle F_{\omega},x\rangle_{\mathcal{A}} d\mu(w)\leq B\langle x,x\rangle_{\mathcal{A}}, \quad  x\in \mathcal{H}.
\end{equation}
\end{itemize}
\end{definition}
\begin{definition}\cite{Ros1}
	Let $\mathcal{H}$ be a Hilbert $\mathcal{A}$-module and $(\Omega , \mu)$ a measure space. A mapping $F: \Omega \longrightarrow \mathcal{H}$ is called a $\ast$-integral frame associted to $(\Omega , \mu)$ if: \\
	\begin{itemize}
		\item [$\centerdot$ ] For all $x\in \mathcal{H}$, $w \longrightarrow \langle x,F_{\omega}\rangle_{\mathcal{A}} $ is measurable function on $\Omega$,
		\item [$\centerdot$ ] there exist two non-zero elements $A$, $B$ in $\mathcal{A}$ such that,
		\begin{equation}
		A\langle x,x\rangle_{\mathcal{A}} A^{\ast} \leq\int_{\Omega}\langle x,F_{\omega}\rangle_{\mathcal{A}} \langle F_{\omega},x\rangle_{\mathcal{A}} d\mu(w)\leq B\langle x,x\rangle_{\mathcal{A}} B^{\ast}, \quad  x\in \mathcal{H}.
		\end{equation}
	\end{itemize}
\end{definition}

\section{Main Results}

\begin{definition}
Let $\mathcal{H}$ be a Hilbert $\mathcal{A}$-module and $(\Omega , \mu)$ a measure space. A $C$-controlled integral frame in $C^{\ast}$-module $\mathcal{H}$ is a map $F: \Omega \longrightarrow \mathcal{H}$ such that there exist $0<A\leq B<\infty$ such that,
\begin{equation}\label{eqd1}
A\langle x,x\rangle_{\mathcal{A}} \leq\int_{\Omega}\langle x,F_{\omega}\rangle_{\mathcal{A}} \langle C F_{\omega},x\rangle_{\mathcal{A}} d\mu(w)\leq B\langle x,x\rangle_{\mathcal{A}}, \quad  x\in \mathcal{H}.
\end{equation}
The elements $A$ and $B$ are called the $C$-controlled integral frame bounds.\\
If $A=B$, we call this a $C$-controlled integral tight frame.\\
If $A=B=1$, it's called a $C$-controlled integral parseval frame.\\
If only the right hand inequality of \eqref{eqd1} is satisfied, we call $F$ a $C$-controlled integral Bessel mapping with bound $B$.
\end{definition}
\begin{example}
	Let $\mathcal{H}=\left\{ X=\left( 
	\begin{array}{ccc}
	a & 0 & 0 \\ 
	0 & 0 & b%
	\end{array}%
	\right) \text{ / }a,b\in 
	%TCIMACRO{\U{2102} }%
	%BeginExpansion
	\mathbb{C}
	%EndExpansion
	\right\} $, \\
	and $\mathcal{A=}\left\{ \left( 
	\begin{array}{cc}
	x & 0 \\ 
	0 & y%
	\end{array}%
	\right) \text{ / }x,y\in 
	%TCIMACRO{\U{2102} }%
	%BeginExpansion
	\mathbb{C}
	%EndExpansion
	\right\} $ which is a $C^{\ast}$-algebra.\\
	 We define the inner product :
	
	\[
	\begin{array}{ccc}
	\mathcal{H}\times \mathcal{H} & \rightarrow  & \mathcal{A} \\ 
	(A,B) & \mapsto  & A(\overline{B})^{t}%
	\end{array}%
	\]\\
	This inner product makes $\mathcal{H}$ a $C^{\ast}$-module over  $\mathcal{A}$.\\
	Let $C$ be an operator defined by,
	
	\begin{align*}
	C: \mathcal{H}   &\longrightarrow    \mathcal{H} \\
	X&\longrightarrow \alpha X
	\end{align*}
	where $\alpha$ is a reel number strictly greater than zero.\\
	It's clair that  $	C \in Gl^{+}(\mathcal{H})$.\\
	Let $\Omega = [0,1]$ endewed with the lebesgue's measure. It's clear that a measure space.\\
	We consider :
	\begin{align*}
	F:\qquad [0,1] &\longrightarrow \mathcal{H}\\
	w&\longrightarrow F_{w}=\left( 
	\begin{array}{cccc}
	w & 0 & 0 \\ 
	0 & 0 & \frac{w}{2}%
	\end{array}%
	\right).
	\end{align*}

	In addition, for $X\in \mathcal{H}$, we have,
	\begin{align*}
	\int_{\Omega}\langle X,F_{w}\rangle_{\mathcal{A}} \langle CF_{w},X\rangle_{\mathcal{A}} {\mathcal{A}}d\mu(\omega)&=\int_{\Omega}\alpha w^{2}\left( 
	\begin{array}{cccc}
	|a|^{2} & 0 \\ 
	0 & \frac{|b|^{2}}{4}%
	\end{array}%
	\right)d\mu(\omega)\\
	&=\frac{\alpha}{3}\left( 
	\begin{array}{cccc}
	|a|^{2} & 0 \\ 
	0 & \frac{|d|^{2}}{4}%
	\end{array}%
	\right).
	\end{align*}
	It's clear that,
	\begin{align*}
	\frac{1}{4}\|X\|_{\mathcal{A}}^{2}\leq 
	\left( 
	\begin{array}{cccc}
	|a|^{2} & 0  \\ 
	0 & \frac{|b|^{2}}{4} %
	\end{array}%
	\right)\leq 
	\left( 
	\begin{array}{cccc}
	|a|^{2} & 0  \\ 
	0 & |b|^{2} %
	\end{array}%
	\right)=\|X\|_{\mathcal{A}}^{2}.
	\end{align*}
	Then we have
	\begin{equation*}
	\frac{\alpha}{12}\|X\|_{\mathcal{A}}^{2}\leq 	\int_{\Omega}\langle X,F_{w}\rangle_{\mathcal{A}} \langle CF_{w},X\rangle_{\mathcal{A}}d\mu(\omega)\leq \frac{\alpha}{3}\|X\|_{\mathcal{A}}^{2}.
	\end{equation*}
	Which show that $F$ is a $C$-controlled integral frame for the $C^{\ast}$-module $\mathcal{H}.$ 
	\end{example}

\begin{definition}
	Let $F$ be a $C$-controlled integral frame for $\mathcal{H}$ associted to $(\Omega , \mu)$. We define the frame operator  $S_{C} : \mathcal{H} \longrightarrow \mathcal{H}$ for $F$ by,
	\begin{equation*}
	S_{C}x=\int_{\Omega}\langle x,F_{\omega}\rangle_{\mathcal{A}} CF_{\omega}  d\mu(\omega), \quad x\in \mathcal{H}.
	 \end{equation*}
\end{definition}
\begin{proposition}
The frame operator $S_{C}$ is positive, selfadjoint, bounded and invertible.
\end{proposition}
\begin{proof} For all $x\in \mathcal{H}$, by lemma \eqref{l4}, we have,
\begin{equation*}
\langle S_{C}x,x\rangle_{\mathcal{A}} = \langle \int_{\Omega}\langle x,F_{\omega}\rangle_{\mathcal{A}} CF_{\omega}  d\mu(\omega),x\rangle_{\mathcal{A}}=\int_{\Omega}\langle x,F_{\omega}\rangle_{\mathcal{A}} \langle CF_{\omega}  ,x\rangle_{\mathcal{A}} d\mu(\omega).
\end{equation*}
By left hand of inequality \eqref{eqd1}, we have,
\begin{equation*}
0\leq A\langle x,x\rangle_{\mathcal{A}} \leq \langle S_{C}x,x\rangle_{\mathcal{A}}.
\end{equation*}
Then $S_{C}$ is a positive operator, also, it's sefladjoint.\\
From \eqref{eqd1}, we have,
\begin{equation*}
A\langle x,x\rangle_{\mathcal{A}} \leq\langle S_{C}x,x\rangle_{\mathcal{A}}\leq B\langle x,x\rangle_{\mathcal{A}}, \quad  x\in \mathcal{H}.
\end{equation*}
So,
\begin{equation*}
A.I \leq S_{C}\leq  B.I
\end{equation*}
Then $S_{C}$ is a bounded operator.\\
Moreover, 
\begin{equation*}
0 \leq I-B^{-1}S_{C} \leq \frac{B-A}{B}.I,
\end{equation*} 
Consequently,
\begin{equation*}
\|I-B^{-1}S_{C} \|=\underset{x \in \mathcal{H}, \|x\|=1}{\sup}\|\langle(I-B^{-1}S_{C})x,x\rangle_{\mathcal{A}} \|\leq \frac{B-A}{B}<1.
\end{equation*}
The Theorem \ref{t0} shows that  $S_{C}$ is invertible.
\end{proof}
\begin{corollary}
	Let $\mathcal{H}$ be a Hilbert $\mathcal{A}$-module and $(\Omega,\mu)$ be a measure space. Let $F : \Omega \longrightarrow \mathcal{H}$ be a mapping. Assume that $S$ is the frame operator for $F$. Then the following statements are equivalent :
	\begin{itemize}
		\item [$(1)$]$F$ is an integral frame  associted to $(\Omega, \mu)$ with integral frame bounds $A$ and $B$.
		\item [$(2)$] We have $A.I \leq S \leq B.I$
	\end{itemize}
\end{corollary}
\begin{proof}
$(1) \Longrightarrow(2)$ Let $F$ be an integral frame associted to $(\Omega, \mu)$ with integral frames bounds $A$ and $B$, then, 
\begin{equation*}
A\langle x,x\rangle_{\mathcal{A}} \leq\int_{\Omega}\langle x,F_{\omega}\rangle_{\mathcal{A}} \langle F_{\omega},x\rangle_{\mathcal{A}} d\mu(w)\leq B\langle x,x\rangle_{\mathcal{A}}, \quad  x\in \mathcal{H}.
\end{equation*}
Since, 
\begin{equation}
 Sx=\int_{\Omega}\langle x,F_{\omega}\rangle_{\mathcal{A}} F_{\omega}d\mu(\omega).
 \end{equation}
 We have,
\begin{equation*}
\langle Sx,x\rangle_{\mathcal{A}} = \langle \int_{\Omega}\langle x,F_{\omega}\rangle_{\mathcal{A}} F_{\omega}d\mu(\omega),x\rangle_{\mathcal{A}} = \int_{\Omega}\langle x,F_{\omega}\rangle_{\mathcal{A}} \langle F_{\omega},x\rangle_{\mathcal{A}} d\mu(\omega),
\end{equation*}
then
\begin{equation*}
\langle Ax,x\rangle_{\mathcal{A}} \leq\langle Sx,x\rangle_{\mathcal{A}} \leq \langle Bx,x\rangle_{\mathcal{A}}, \quad  x\in \mathcal{H}.
\end{equation*}
So,
\begin{equation*}
A.I \leq S \leq B.I.
\end{equation*}
$(2) \Longrightarrow(1)$ 
Let $x\in \mathcal{H}$, then,
\begin{equation}\label {p1}
\|\int_{\Omega}\langle x,F_{\omega}\rangle_{\mathcal{A}} \langle F_{\omega},x\rangle_{\mathcal{A}} d\mu(w)\|=\|\langle S_{C}x,x\rangle_{\mathcal{A}} \|\leq \|S_{C}x\|\|x\|\leq B\|x\|^{2}
\end{equation}
Also,
\begin{equation}\label{p2}
\|\langle S_{C}x,x\rangle_{\mathcal{A}} \|\geq \|\langle Ax,x\rangle_{\mathcal{A}}\| \geq A\|x\|^{2}
\end{equation}
By \eqref{p1} and \eqref{p2} we obtain
\begin{equation*}
A\|x\|^{2} \leq \|\int_{\Omega}\langle x,F_{\omega}\rangle_{\mathcal{A}} \langle F_{\omega},x\rangle_{\mathcal{A}} d\mu(w)\| \leq B\|x\|^{2}
\end{equation*}
Which ends the proof.
\end{proof}
\begin{theorem}
	Let $\mathcal{H}$ be a Hilbert $\mathcal{A}$-module, $(\Omega , \mu)$ be a measure space and let $F$ be a mapping over $\Omega$ to $\mathcal{H}$, then $F$ is an integral frame associted to $(\Omega , \mu)$ if and only if there exist $0<A\leq B<\infty$ such that,
	\begin{equation}\label{eq01t1}
	A\|x\|^{2} \leq\|\int_{\Omega}\langle x,F_{\omega}\rangle_{\mathcal{A}} \langle F_{\omega},x\rangle_{\mathcal{A}} d\mu(w)\|^{2}\leq B\|x\|^{2} \quad  x\in \mathcal{H}.
	\end{equation}
\end{theorem}
\begin{proof}
Let $F$ be an integral frame associted to $(\Omega , \mu)$ with bounds $A$ and $B$, then,
 \begin{equation}
 A\langle x,x\rangle_{\mathcal{A}} \leq\int_{\Omega}\langle x,F_{\omega}\rangle_{\mathcal{A}} \langle F_{\omega},x\rangle_{\mathcal{A}} d\mu(w)\leq B\langle x,x\rangle_{\mathcal{A}}, \quad  x\in \mathcal{H}.
 \end{equation}
 Since the lower and upper bounds are positive then we have,
 \begin{equation*}
 A\|x\|^{2} \leq\|\int_{\Omega}\langle x,F_{\omega}\rangle_{\mathcal{A}} \langle F_{\omega},x\rangle_{\mathcal{A}} d\mu(w)\|^{2}\leq B\|x\|^{2} \quad  x\in \mathcal{H}.
 \end{equation*}
 Conversely, suppose \eqref{eq01t1} holds.
By (\cite{Ros1}, Theorem 2.4), we have, 
 \begin{equation*}
 \|\int_{\Omega}\langle x,F_{\omega}\rangle_{\mathcal{A}} \langle F_{\omega},x\rangle_{\mathcal{A}} d\mu(w)\|^{2}=\|\langle Sx,x\rangle_{\mathcal{A}}\|
 =\|S^{\frac{1}{2}}x,S^{\frac{1}{2}}x\|=\|S^{\frac{1}{2}}x\|^{2}
 \end{equation*}
 Then,
 \begin{equation*}
 A\|x\|^{2} \leq\|S^{\frac{1}{2}}x\|^{2}\leq B\|x\|^{2} \quad  x\in \mathcal{H}.
 \end{equation*}
 By lemma \ref{l2}, there exists $0<m,M$ such that,
 \begin{equation}
 m\langle x,x\rangle_{\mathcal{A}} \leq\int_{\Omega}\langle x,F_{\omega}\rangle_{\mathcal{A}} \langle F_{\omega},x\rangle_{\mathcal{A}} d\mu(w)\leq M\langle x,x\rangle_{\mathcal{A}}, \quad  x\in \mathcal{H}.
 \end{equation}
 which ends the proof. 
\end{proof}

\begin{theorem}
Let $\mathcal{H}$ be a Hilbert $\mathcal{A}$-module, $C\in GL^{+}(\mathcal{H})$ and $(\Omega , \mu)$ a measure space and $F$ be a mapping for $\Omega$ to $\mathcal{H}$. Then $F$ is a $C$-controlled integral frame for $\mathcal{H}$ associted to $(\Omega , \mu)$ if and only if there exist $0<A\leq B<\infty$ such that,
\begin{equation}\label{eq1t1}
A\|x\|^{2} \leq\|\int_{\Omega}\langle x,F_{\omega}\rangle_{\mathcal{A}} \langle C F_{\omega},x\rangle_{\mathcal{A}} d\mu(w)\|^{2}\leq B\|x\|^{2} \quad  x\in \mathcal{H}.
\end{equation}
\end{theorem}
\begin{proof}
$\Longrightarrow$) obvious.\\
$\Longleftarrow$) Supposes there exists $0<A\leq B<\infty$, such that \eqref{eq1t1} holds.\\
On one hand, for all $x\in \mathcal{H}$ we have ,
\begin{align*}
A\|x\|^{2} &\leq\|\int_{\Omega}\langle x,F_{\omega}\rangle_{\mathcal{A}} \langle C F_{\omega},x\rangle_{\mathcal{A}} d\mu (\omega)\|\\
&= \|\langle S_{C}x,x\rangle_{\mathcal{A}}\|\\
&= \|\langle S_{C}^{\frac{1}{2}}x,S_{C}^{\frac{1}{2}}x\rangle_{\mathcal{A}}\|\\
&= \|S_{C}^{\frac{1}{2}}x\|^{2}.
\end{align*}
By lemma \ref{l2}, there exist $0< m$ such that,
\begin{equation}\label{tt1}
m \langle x,x\rangle_{\mathcal{A}} \leq \langle S_{C}^{\frac{1}{2}}x,S_{C}^{\frac{1}{2}}x\rangle_{\mathcal{A}}=\langle S_{C}x,x\rangle_{\mathcal{A}}. 
\end{equation}
On other hand, for all $x\in \mathcal{H}$ we have ,
\begin{align*}
B\|x\|^{2}&\geq \|\int_{\Omega}\langle x,F_{\omega}\rangle_{\mathcal{A}} \langle C F_{\omega},x\rangle_{\mathcal{A}} d\mu(w)\|^{2}\\
&= \|\langle S_{C}x,x\rangle_{\mathcal{A}}\|\\
&= \|\langle S_{C}^{\frac{1}{2}}x,S_{C}^{\frac{1}{2}}x\rangle_{\mathcal{A}}\|\\
&= \|S_{C}^{\frac{1}{2}}x\|^{2}.
\end{align*}
By lemma \ref{l2}, there exist $0< m^{'}$ such that,
\begin{equation}\label{tt2}
\langle S_{C}^{\frac{1}{2}}x,S_{C}^{\frac{1}{2}}x\rangle_{\mathcal{A}}=\langle S_{C}x,x\rangle_{\mathcal{A}} \leq m^{'}\langle x,x\rangle_{\mathcal{A}}. 
\end{equation}
From \eqref{tt1} and \eqref{tt2}, we conclude that $F$ is a $C$-controlled integral frame.
\end{proof}
\begin{proposition}
Let $C\in GL^{+}(\mathcal{H})$ and $F$ be a $C$-controlled integral frame for $\mathcal{H}$ associted to $(\Omega , \mu)$ with bounds $A$ and $B$. Then $F$ is an integral frame for $\mathcal{H}$ associted to $(\Omega , \mu)$ with bounds $A\|C^{\frac{1}{2}}\|^{-2}$ and $B\|C^{\frac{-1}{2}}\|^{2}$.
\end{proposition}
\begin{proof}
Let $F$ be a $C$-controlled integral frame for $\mathcal{H}$ associted to $(\Omega , \mu)$, with bounds $A$ and $B$.\\
On one hand we have,
\begin{align*}
A\langle x,x\rangle_{\mathcal{A}} &\leq \langle S_{C}x,x\rangle_{\mathcal{A}} \\
&=\langle CSx,x\rangle_{\mathcal{A}}\\
&= \langle C^{\frac{1}{2}}Sx,C^{\frac{1}{2}}x\rangle_{\mathcal{A}}  \\
&\leq \|C^{\frac{1}{2}}\|^{2}\langle Sx,x\rangle_{\mathcal{A}}  \\
\end{align*}
So,
\begin{equation}\label{p10}
A\|C^{\frac{1}{2}}\|^{-2}\langle x,x\rangle_{\mathcal{A}} \leq\int_{\Omega}\langle x,F_{\omega}\rangle_{\mathcal{A}} \langle F_{\omega},x\rangle_{\mathcal{A}} d\mu(w)
\end{equation}
On other hand, for all $x\in \mathcal{H}$, we have,
\begin{align*}
\int_{\Omega}\langle x,F_{\omega}\rangle_{\mathcal{A}} \langle F_{\omega},x\rangle_{\mathcal{A}} d\mu(w)&=\langle Sx,x\rangle_{\mathcal{A}} \\
&=\langle C^{-1}CSx,x\rangle_{\mathcal{A}} \\
&=\langle (C^{-1}CS)^{\frac{1}{2}}x,(C^{-1}CS)^{\frac{1}{2}}x\rangle_{\mathcal{A}}\\
&=\|(C^{-1}CS)^{\frac{1}{2}}x\|^{2}\\
&\leq \|C^{\frac{-1}{2}}\|^{2} \|(CS)^{\frac{1}{2}}x\|^{2}\\
&= \|C^{\frac{-1}{2}}\|^{2}\langle (CS)^{\frac{1}{2}}x,(CS)^{\frac{1}{2}}x\rangle_{\mathcal{A}}\\
&=\|C^{\frac{-1}{2}}\|^{2}\langle (S_{C})^{\frac{1}{2}}x,(S_{C})^{\frac{1}{2}}x\rangle_{\mathcal{A}}\\
&=\|C^{\frac{-1}{2}}\|^{2}\langle S_{C}x,x\rangle_{\mathcal{A}}\\
&\leq \|C^{\frac{-1}{2}}\|^{2}B\langle x,x\rangle_{\mathcal{A}}.
\end{align*}
Then,
\begin{equation}\label{pp10}
\int_{\Omega}\langle x,F_{\omega}\rangle_{\mathcal{A}} \langle F_{\omega},x\rangle_{\mathcal{A}} d\mu(w)\leq \|C^{\frac{-1}{2}}\|^{2}B\langle x,x\rangle_{\mathcal{A}}.
\end{equation}
From \eqref{p10} and \eqref{pp10} we conclude that $F$ is an integral frame $\mathcal{H}$ associted to $(\Omega , \mu)$ with bounds $A\|C^{\frac{1}{2}}\|^{-2}$ and $B\|C^{\frac{-1}{2}}\|^{2}$.
\end{proof}
\begin{proposition}
Let $C\in GL^{+}(\mathcal{H})$ and $F$ be an integral frame for $\mathcal{H}$ associted to $(\Omega , \mu)$ with bounds $A$ and $B$. Then $F$ is a $C$-controlled integral frame for $\mathcal{H}$ associted to $(\Omega , \mu)$ with bounds $A\|C^{\frac{-1}{2}}\|^{2}$ and $B\|C^{\frac{1}{2}}\|^{2}$.
\end{proposition}
\begin{proof}
Let $F$ be an integral frame for $\mathcal{H}$ associted to $(\Omega , \mu)$ with bounds $A$ and $B$. Then for all $x\in \mathcal{H}$, we have
\begin{align*}
A\langle x, x\rangle_{\mathcal{A}}&\leq \langle Sx,x\rangle_{\mathcal{A}} \\
&=\langle C^{-1}CSx,x\rangle_{\mathcal{A}} \\
&=\langle (C^{-1}CS)^{\frac{1}{2}}x,(C^{-1}CS)^{\frac{1}{2}}x\rangle_{\mathcal{A}}\\
&=\|(C^{-1}CS)^{\frac{1}{2}}x\|^{2}\\
&\leq \|C^{\frac{-1}{2}}\|^{2} \|(CS)^{\frac{1}{2}}x\|^{2}\\
&= \|C^{\frac{-1}{2}}\|^{2}\langle (CS)^{\frac{1}{2}}x,(CS)^{\frac{1}{2}}x\rangle_{\mathcal{A}}\\
&=\|C^{\frac{-1}{2}}\|^{2}\langle (S_{C})^{\frac{1}{2}}x,(S_{C})^{\frac{1}{2}}x\rangle_{\mathcal{A}}\\
&=\|C^{\frac{-1}{2}}\|^{2}\langle S_{C}x,x\rangle_{\mathcal{A}}\\
\end{align*}
So,
\begin{equation}
A\|C^{\frac{-1}{2}}\|^{-2}\langle x, x\rangle_{\mathcal{A}} \leq \langle S_{C}x,x\rangle_{\mathcal{A}}
\end{equation}
Hence, for all $x\in \mathcal{H}$, we have,
\begin{align*}
\langle S_{C}x,x\rangle_{\mathcal{A}}&=\langle CSx,x\rangle_{\mathcal{A}}\\
&=\langle C^{\frac{1}{2}}Sx,C^{\frac{1}{2}}x\rangle_{\mathcal{A}}\\
&\leq \|C^{\frac{1}{2}}\|^{2}\langle Sx,x\rangle_{\mathcal{A}} \\
&\leq \|C^{\frac{1}{2}}\|^{2}B\langle x,x\rangle_{\mathcal{A}}. 
\end{align*}
Therefore we conclude that $F$ is a $C$-controlled integral frame $\mathcal{H}$ associted to $(\Omega , \mu)$ with bounds $A\|C^{\frac{-1}{2}}\|^{-2}$ and $B\|C^{\frac{1}{2}}\|^{2}$.
\end{proof}
\begin{theorem}
Let $\mathcal{H}$ be a Hilbert $\mathcal{A}$-module and $(\Omega , \mu)$ a measure space. Let $F$ be a $C$-controlled integral frame for $\mathcal{H}$ associted to $(\Omega , \mu)$ with the frame operator $S_{C}$ and bounds $A$ and $B$. Let $K\in End_{\mathcal{A}}^{\ast}(\mathcal{H})$ a surjective operator such that $KC=CK$. Then $KF$ is a $C$-controlled integral frame for $\mathcal{H}$ with the operator frame $KS_{C}K^{\ast}$.
\end{theorem}
\begin{proof}
Let $F$ be a $C$-controlled integral frame for $\mathcal{H}$ associted to $(\Omega , \mu)$, then,
\begin{equation*}
A\langle K^{\ast}x,K^{\ast}x\rangle_{\mathcal{A}} \leq\int_{\Omega}\langle K^{\ast}x,F_{\omega}\rangle_{\mathcal{A}} \langle C F_{\omega},K^{\ast}x\rangle_{\mathcal{A}} d\mu(w)\leq B\langle K^{\ast}x,K^{\ast}x\rangle_{\mathcal{A}}, \quad  x\in \mathcal{H}.
\end{equation*}
By lemma \ref{l1} and lemma \ref{l3}, we obtain,
\begin{equation*}
A\|(KK^{\ast})^{-1}\|^{-1}\langle x,x\rangle_{\mathcal{A}} \leq\int_{\Omega}\langle x,KF_{\omega}\rangle_{\mathcal{A}} \langle CK F_{\omega},x\rangle_{\mathcal{A}} d\mu(w)\leq B\|K\|^{2}\langle x,x\rangle_{\mathcal{A}}, \quad  x\in \mathcal{H}.
\end{equation*}
which shows that $KF$ is a $C$-controlled integral operator.\\
Moreover, by lemma \ref{l4}, we have,
\begin{equation*}
KS_{C}K^{\ast}x=K\int_{\Omega}\langle K^{\ast}x,F_{\omega}\rangle_{\mathcal{A}} CF_{\omega} d\mu (\omega)= \int_{\Omega}\langle x,KF_{\omega}\rangle_{\mathcal{A}} CKF_{\omega} d\mu (\omega),
\end{equation*}
which ends the proof.
\end{proof}

\section{Controlled $\ast$-integral frames}
\begin{definition}
	Let $\mathcal{H}$ be a Hilbert $\mathcal{A}$-module and $(\Omega , \mu)$ a measure space. A $C$-controlled $\ast$-integral frame in $C^{\ast}$-module $\mathcal{H}$ is a map $F: \Omega \longrightarrow \mathcal{H}$ such that there exist two strictly nonzero elements A,B in A such that,
	\begin{equation}\label{eqd2}
	A\langle x,x\rangle_{\mathcal{A}} A^{\ast} \leq\int_{\Omega}\langle x,F_{\omega}\rangle_{\mathcal{A}} \langle C F_{\omega},x\rangle_{\mathcal{A}} d\mu(w)\leq B\langle x,x\rangle_{\mathcal{A}} B^{\ast}, \quad  x\in \mathcal{H}.
	\end{equation}
	The elements $A$ and $B$ are called the $C$-controlled $\ast$-integral frame bounds.\\
	If $A=B$, we call this a $C$-controlled $\ast$-integral tight frame.\\
	If $A=B=1$, it's called a $C$-controlled $\ast$-integral parseval frame.\\
	If only the right hand inequality of \eqref{eqd2} is satisfied, we call $F$ a $C$-controlled $\ast$-integral Bessel mapping with bound $B$.
\end{definition}
\begin{example} 
	Let $\mathcal{H}=\mathcal{A}=\{(a_{n})_{n\in \mathbb{N}}\subset \mathbb{C}, \quad \sum_{n\geq 0}|a_{n}|< \infty\}$ \\
	Endeweed with the product and the inner product defined as follow.
	\[
	\begin{array}{ccc}
	\mathcal{A}\times \mathcal{A} & \rightarrow  & \mathcal{A} \\ 
	((a_{n})_{n\in \mathbb{N}},(b_{n})_{n\in \mathbb{N}}) & \mapsto  & (a_{n})_{n\in \mathbb{N}}.(b_{n})_{n\in \mathbb{N}}=(a_{n}b_{n})_{n\in \mathbb{N}}%
	\end{array}%
	\]\\
	and
		\[
	\begin{array}{ccc}
	\mathcal{H}\times \mathcal{H} & \rightarrow  & \mathcal{A} \\ 
	((a_{n})_{n\in \mathbb{N}},(b_{n})_{n\in \mathbb{N}}) & \mapsto  & \langle (a_{n})_{n\in \mathbb{N}},(b_{n})_{n\in \mathbb{N}}\rangle_{\mathcal{A}} =(a_{n}\overline{b_{n}})_{n\in \mathbb{N}}%
	\end{array}%
	\]\\
	Let $\Omega = [0,+\infty[$ endewed with the lebesgue's measure wich's a measure space.
	\begin{align*}
	F:\qquad [0,+\infty[ &\longrightarrow \mathcal{H}\\
	w&\longrightarrow F_{w}=(F^{w}_{n})_{n\in \mathbb{N}},
	\end{align*}
	where 
	\begin{equation*}
	F^{w}_{n} = \frac{1}{n+1} \quad if \quad n=[w] \qquad and \qquad F^{w}_{n} =0 \quad elsewhere,
	\end{equation*}
	where $[w]$ is the whole part of $w$.\\
	On the other hand, we consider the measure space $(\Omega,\mu)$, where $\mu$ is the lebesgue measure restricted to $ [0,+\infty[$, and the operator,
	\begin{align*}
C:\qquad\mathcal{H}&\longrightarrow \mathcal{H}  \\
(a_{n})_{n\in \mathbb{N}}	&\longrightarrow (\alpha a_{n})_{n\in \mathbb{N}},
	\end{align*}
	where $\alpha$ is a strictly positive real number.\\
	It's clear that $C$ is an invertible and both operators, and $C$,$C^{-1}$ are bounded. \\
	So,
	\begin{align*}
&\int_{\Omega}\langle (a_{n})_{n\in \mathbb{N}},F_{w}\rangle_{\mathcal{A}}\langle CF_{w},(a_{n})_{n\in \mathbb{N}}\rangle_{\mathcal{A}}d\mu(w)\\
&=\int_{1}^{+\infty}(0,0,...,\frac{a_{[w]}}{[w]+1},0,...)\alpha (0,0,...,\overline{\frac{a_{[w]}}{[w]+1}},0,...)d\mu (w)\\
&=\alpha \sum_{p=0}^{+\infty}\int_{p}^{p+1}(0,0,...,\frac{|a_{[w]}|^{2}}{([w]+1)^{2 }},0,...)d\mu (w)\\
&=\alpha \sum_{p=0}^{+\infty}(0,0,...,\frac{|a_{[p]}|^{2}}{(p+1)^{2 }},0,...)\\
&=\alpha (\frac{|a_{n}|^{2}}{(n+1)^{2 }})_{n\in \mathbb{N}}\\
&=\sqrt{\alpha}(1,\frac{1}{2},\frac{1}{3},...,\frac{1}{n},...)\langle (a_{n})_{n\in \mathbb{N}},(a_{n})_{n\in \mathbb{N}}\rangle_{\mathcal{A}}\sqrt{\alpha}(1,\frac{1}{2},\frac{1}{3},...,\frac{1}{n},...).
\end{align*}
Which shows that $F$ is a $C$-controlled $\ast$-integral tight frame for $\mathcal{H}$ with bound $A=\sqrt{\alpha}(1,\frac{1}{2},\frac{1}{3},...,\frac{1}{n},...) \in \mathcal{A}$
\end{example}
\begin{definition}
	Let $F$ be a $C$-controlled $\ast$-integral frame for $\mathcal{H}$ associted to $(\Omega , \mu)$. We define the frame operator  $S_{C} : \mathcal{H} \longrightarrow \mathcal{H}$ for $F$ by,
	\begin{equation*}
	S_{C}x=\int_{\Omega}\langle x,F_{\omega}\rangle_{\mathcal{A}} CF_{\omega}  d\mu(\omega), \quad x\in \mathcal{H}.
	\end{equation*}
\end{definition}
\begin{proposition}
	The frame operator $S_{C}$ is positive, selfadjoint, bounded and invertible.
\end{proposition}
\begin{proof} For all $x\in \mathcal{H}$, by lemma \eqref{l4}, we have,
	\begin{equation*}
	\langle S_{C}x,x\rangle_{\mathcal{A}} = \langle \int_{\Omega}\langle x,F_{\omega}\rangle_{\mathcal{A}} CF_{\omega}  d\mu(\omega),x\rangle_{\mathcal{A}}=\int_{\Omega}\langle x,F_{\omega}\rangle_{\mathcal{A}} \langle CF_{\omega}  ,x\rangle_{\mathcal{A}} d\mu(\omega).
	\end{equation*}
	By left hand of inequality \eqref{eqd2}, we deduce that $S_{C}$ is a positive operator, also, it's sefladjoint.\\
	From \eqref{eqd2}, we have,
	\begin{equation*}\label{111}
	A\langle x,x\rangle_{\mathcal{A}} \leq\langle S_{C}x,x\rangle_{\mathcal{A}}\leq B\langle x,x\rangle_{\mathcal{A}}, \quad  x\in \mathcal{H}.
	\end{equation*}
	The Theorem 2.5 in \cite{MNAZ} shows that  $S_{C}$ is invertible.
\end{proof}
\begin{proposition}
	Let $C\in GL^{+}(\mathcal{H})$ and $F$ be a $C$-controlled $\ast$-integral frame for $\mathcal{H}$ associted to $(\Omega , \mu)$ with bounds $A$ and $B$. Then $F$ is a $\ast$-integral frame $\mathcal{H}$ associted to $(\Omega , \mu)$ with bounds $\|C^{\frac{1}{2}}\|^{-1}A$ and $\|C^{\frac{-1}{2}}\|B$
\end{proposition}
\begin{proof}
	Let $F$ be a $C$-controlled $\ast$-integral frame for $\mathcal{H}$ associted to $(\Omega , \mu)$, with bounds $A$ and $B$.\\
	On one hand we have 
	\begin{align*}
	A\langle x,x\rangle_{\mathcal{A}} A^{\ast} &\leq \langle S_{C}x,x\rangle_{\mathcal{A}} \\
	&=\langle CSx,x\rangle_{\mathcal{A}}\\
	&= \langle C^{\frac{1}{2}}Sx,C^{\frac{1}{2}}x\rangle_{\mathcal{A}}  \\
	&\leq \|C^{\frac{1}{2}}\|^{2}\langle Sx,x\rangle_{\mathcal{A}}.  \\
	\end{align*}
	So,
	\begin{equation}\label{p10}
	(\|C^{\frac{1}{2}}\|^{-1}A)\langle x,x\rangle_{\mathcal{A}} (\|C^{\frac{1}{2}}\|^{-1}A)^{\ast} \leq\int_{\Omega}\langle x,F_{\omega}\rangle_{\mathcal{A}} \langle F_{\omega},x\rangle_{\mathcal{A}} d\mu(w).
	\end{equation}
	On other hand, for all $x\in \mathcal{H}$, we have,
	\begin{align*}
	\int_{\Omega}\langle x,F_{\omega}\rangle_{\mathcal{A}} \langle F_{\omega},x\rangle_{\mathcal{A}} d\mu(w)&=\langle Sx,x\rangle_{\mathcal{A}} \\
	&=\langle C^{-1}CSx,x\rangle_{\mathcal{A}} \\
	&=\langle (C^{-1}CS)^{\frac{1}{2}}x,(C^{-1}CS)^{\frac{1}{2}}x\rangle_{\mathcal{A}}\\
	&=\|(C^{-1}CS)^{\frac{1}{2}}x\|^{2}\\
	&\leq \|C^{\frac{-1}{2}}\|^{2} \|(CS)^{\frac{1}{2}}x\|^{2}\\
	&= \|C^{\frac{-1}{2}}\|^{2}\langle (CS)^{\frac{1}{2}}x,(CS)^{\frac{1}{2}}x\rangle_{\mathcal{A}}\\
	&=\|C^{\frac{-1}{2}}\|^{2}\langle (S_{C})^{\frac{1}{2}}x,(S_{C})^{\frac{1}{2}}x\rangle_{\mathcal{A}}\\
	&=\|C^{\frac{-1}{2}}\|^{2}\langle S_{C}x,x\rangle_{\mathcal{A}}\\
	&\leq \|C^{\frac{-1}{2}}\|^{2}B\langle x,x\rangle_{\mathcal{A}} B^{\ast}.
	\end{align*}
	Then,
	\begin{equation}\label{pp10}
	\int_{\Omega}\langle x,F_{\omega}\rangle_{\mathcal{A}} \langle F_{\omega},x\rangle_{\mathcal{A}} d\mu(w)\leq (\|C^{\frac{-1}{2}}\|B)\langle x,x\rangle_{\mathcal{A}} (\|C^{\frac{-1}{2}}\|B)^{\ast}.
	\end{equation}
	From \eqref{p10} and \eqref{pp10} we conclude that $F$ is a $\ast$-integral frame $\mathcal{H}$ associted to $(\Omega , \mu)$ with bounds $A\|C^{\frac{1}{2}}\|^{-2}$ and $B\|C^{\frac{-1}{2}}\|^{2}$.
\end{proof}
\begin{proposition}
	Let $C\in GL^{+}(\mathcal{H})$ and $F$ be an $\ast$-integral frame for $\mathcal{H}$ associted to $(\Omega , \mu)$ with bounds $A$ and $B$. Then $F$ is a $C$-controlled $\ast$-integral frame for $\mathcal{H}$ associted to $(\Omega , \mu)$ with bounds $\|C^{\frac{-1}{2}}\|^{-1}A$ and $\|C^{\frac{1}{2}}\|B$.
\end{proposition}
\begin{proof}
	Let $F$ be an integral frame for $\mathcal{H}$ associted to $(\Omega , \mu)$ with bounds $A$ and $B$.  Then for all $x\in \mathcal{H}$, we have
	\begin{align*}
	A\langle x, x\rangle_{\mathcal{A}} A^{\ast}&\leq \langle Sx,x\rangle_{\mathcal{A}} \\
	&=\langle C^{-1}CSx,x\rangle_{\mathcal{A}} \\
	&=\langle (C^{-1}CS)^{\frac{1}{2}}x,(C^{-1}CS)^{\frac{1}{2}}x\rangle_{\mathcal{A}}\\
	&=\|(C^{-1}CS)^{\frac{1}{2}}x\|^{2}\\
	&\leq \|C^{\frac{-1}{2}}\|^{2} \|(CS)^{\frac{1}{2}}x\|^{2}\\
	&= \|C^{\frac{-1}{2}}\|^{2}\langle (CS)^{\frac{1}{2}}x,(CS)^{\frac{1}{2}}x\rangle_{\mathcal{A}}\\
	&=\|C^{\frac{-1}{2}}\|^{2}\langle (S_{C})^{\frac{1}{2}}x,(S_{C})^{\frac{1}{2}}x\rangle_{\mathcal{A}}\\
	&=\|C^{\frac{-1}{2}}\|^{2}\langle S_{C}x,x\rangle_{\mathcal{A}}.\\
	\end{align*}
	So,
	\begin{equation}
	(\|C^{\frac{-1}{2}}\|^{-1}A)\langle x, x\rangle_{\mathcal{A}} (\|C^{\frac{-1}{2}}\|^{-1}A)^{\ast}\leq \langle S_{C}x,x\rangle_{\mathcal{A}}
	\end{equation}
	Hence, for all $x\in \mathcal{H}$,
	\begin{align*}
	\langle S_{C}x,x\rangle_{\mathcal{A}}&=\langle CSx,x\rangle_{\mathcal{A}}\\
	&=\langle C^{\frac{1}{2}}Sx,C^{\frac{1}{2}}x\rangle_{\mathcal{A}}\\
	&\leq \|C^{\frac{1}{2}}\|^{2}\langle Sx,x\rangle_{\mathcal{A}} \\
	&\leq \|C^{\frac{1}{2}}\|^{2}B\langle x,x\rangle_{\mathcal{A}} B^{\ast} \\
	&= (\|C^{\frac{1}{2}}\|B)\langle x,x\rangle_{\mathcal{A}} (\|C^{\frac{1}{2}}\|B)^{\ast}.
	\end{align*}
	Therefore we conclude that $F$ is a $C$-controlled $\ast$-integral frame $\mathcal{H}$ associted to $(\Omega , \mu)$ with bounds $\|C^{\frac{-1}{2}}\|^{-1}A$ and $\|C^{\frac{1}{2}}\|B$.
\end{proof}
\begin{theorem}
	Let $\mathcal{H}$ be a Hilbert $\mathcal{A}$-module and $(\Omega , \mu)$ a measure space. Let $F$ a $C$-controlled $\ast$-integral frame for $\mathcal{H}$ associted to $(\Omega , \mu)$ with the frame operator $S_{C}$ and bounds $A$ and $B$. Let $K \in End_{\mathcal{A}}^{\ast}(\mathcal{H})$ a surjective operator such that $KC=CK$. Then $KF$ is a $C$-controlled $\ast$-integral frame for $\mathcal{H}$ with the operator frame $KS_{C}K^{\ast}$.
\end{theorem}
\begin{proof}
	For all $x\in \mathcal{H}$ and \eqref{eqd2},we have,
	\begin{equation*}
	A\langle K^{\ast}x,K^{\ast}x\rangle_{\mathcal{A}} A^{\ast} \leq\int_{\Omega}\langle K^{\ast}x,F_{\omega}\rangle_{\mathcal{A}} \langle C F_{\omega},K^{\ast}x\rangle_{\mathcal{A}} d\mu(w)\leq B\langle  K^{\ast}x,K^{\ast}x\rangle_{\mathcal{A}} B^{\ast}, \quad  x\in \mathcal{H}.
	\end{equation*}
	By lemma \ref{l1} and lemma \ref{l3}, we obtain,
	\begin{equation*}
	A\|(KK^{\ast})^{-1}\|^{-1}\langle x,x\rangle_{\mathcal{A}} A^{\ast} \leq\int_{\Omega}\langle x,KF_{\omega}\rangle_{\mathcal{A}} \langle CK F_{\omega},x\rangle_{\mathcal{A}} d\mu(w)\leq B\|K\|^{2}\langle x,x\rangle_{\mathcal{A}} B^{\ast}, \quad  x\in \mathcal{H}.
	\end{equation*}
	which shows that $KF$ is a $C$-controlled $\ast$-integral operator.\\
	Moreover, by lemma \ref{l4}, we have,
	\begin{equation*}
	KS_{C}K^{\ast}x=K\int_{\Omega}\langle K^{\ast}x,F_{\omega}\rangle_{\mathcal{A}} CF_{\omega} d\mu (\omega)= \int_{\Omega}\langle x,KF_{\omega}\rangle_{\mathcal{A}} CKF_{\omega} d\mu (\omega)
	\end{equation*}
	which ends the proof
\end{proof}

\subsection*{Acknowledgment}
The authors would like to express their gratitude to the reviewers for helpful comments and suggestions.
\bibliographystyle{amsplain}

%\vspace{0.1in}
%\hrule width \hsize \kern 1mm
%\hrule width \hsize height 2pt
\end{document}